\documentclass[12pt, a4paper]{amsart}
\newtheorem{prop}{Proposition}[section]
\newtheorem{lema}[prop]{Lemma}
\newtheorem{teo}[prop]{Theorem}
\newtheorem{corolario}[prop]{Corollary}

\newtheorem{remark}[prop]{\sc Remark}

\newtheorem{obs}[prop]{\sc Observation}

\newcommand{\supp}{\mbox{supp}}
\newcommand{\uno}{1\!\!1}

\title[Valuations on star bodies]{Radial continuous rotation invariant valuations on star bodies}

\author{Ignacio Villanueva}

\address{Departamento de An\'alisis Matem\'atico \\
Facultad de Matem\'aticas \\ Universidad Complutense de Madrid \\
Madrid 28040}

\email{ignaciov@mat.ucm.es}

\thanks{Partially supported by  MINECO (grant MTM2011-26912) and  Comunidad de Madrid (grant QUITEMAD+-CM, ref. S2013/ICE-2801)}

\begin{document}

\begin{abstract}
We characterize the positive radial continuous and rotation invariant valuations $V$ defined on the star bodies of $\mathbb R^n$ as the applications on star bodies which admit an integral representation with respect to the Lebesgue measure. That is, $$V(K)=\int_{S^{n-1}}\theta(\rho_K)dm,$$ where $\theta$ is a positive continuous function, $\rho_K$ is the radial function associated to $K$ and $m$ is the Lebesgue measure on $S^{n-1}$. As a corollary, we obtain that every such valuation can be uniformly approximated on bounded sets by a linear combination of dual quermassintegrals. 
\end{abstract}



\maketitle

\section{Introduction}

Valuations can be thought of as a generalization of the notion of measure. Valuations on convex bodies have been studied for a long time now, starting with the solution of Hilbert's Third Problem in 1901. 

Since then until today, valuations and their study have become a most relevant area of study in Convex Geometry. See, for instance, \cite{Alesker}, \cite{Alesker3}, \cite{Klain96}, \cite{Klain97}, \cite{Lu1}, \cite{Lu2}. See also \cite{Ma}, \cite{Ko}, \cite{Ts} for recent developments related to our paper. References in \cite{Lu1}, \cite{Lu2} provide a broad vision of the field. 

In the 1950's, Hadwiger initiated a systematic study of valuations on convex bodies and, in particular, he proved the result which we now know as Hadwiger's Theorem, which characterizes continuous rotation and translation invariant valuations on convex bodies as linear combinations of the quermassintegrals \cite{Had}. In \cite{Alesker}, Alesker studies the  valuations on convex bodies which are only rotation invariant. 

Valuations on convex bodies belong naturally to the Brunn-Minkowski Theory, one of the cornerstones of modern geometry. Brunn-Minkowski Theory has been extended and modified in several directions. One of the main theories derived from it is what we now know as Dual Brunn-Minkowski Theory. In this dual theory convex bodies, Minkowski addition, the Hausdorff metric and mixed volumes are replaced by star bodies, radial addition, radial metric and dual mixed volumes, respectively. Ever since it was initiated in \cite{Lut_mv1}, the dual Brunn-Minkowski theory has been broadly developed and has been successfully applied in several areas. In particular, it played a key role in the solution of the Busemann–-Petty problem \cite{Ga1}, \cite{Ga2}, \cite{Zh}.

The study of valuations on star sets was initiated in \cite{Klain96}, \cite{Klain97}, where Klain studies rotation invariant valuations. The valuations studied in those papers are defined on $L^n$-stars, star sets whose radial function belongs to $L^n(S^{n-1})$.  

A  star body (or in general a star set),  $K\subset \mathbb R^n$ is characterized by its {\em radial function} $\rho_K:S^{n-1}\longrightarrow \mathbb R^+$, which assigns to each direction in $\mathbb R^n$ the length of $K$ along that direction (see Section \ref{sectionnotation} for the proper definitions). 

In this note, we characterize the positive rotation invariant valuations on star bodies which are  continuous with respect to the radial topology. Note that 
$L^n$-stars form a much bigger class than star bodies (star sets whose radial function is continuous). Therefore, our work can be viewed as a generalization of Klain's work in the case of positive valuations.

Our main result says 

\begin{teo}\label{main} 
If $V:\mathcal S_0^n\longrightarrow \mathbb R^+$ is a rotation invariant and radial continuous valuation on the $n$-dimensional star bodies $\mathcal S_0^n$, verifying that $V(\{0\})=0$, then there exists a continuous function $\theta:[0,\infty) \longrightarrow \mathbb R^+$ such that $\theta(0)=0$ and such that, for every $K\in \mathcal S_0^n$,  
$$V(K)=\int_{S^{n-1}} \theta(\rho_K(t)) dm(t),$$
where $m$ is the Lebesgue measure on $S^{n-1}$ and $\rho_K$ is the radial function of $K$. 

Conversely, let $\theta:\mathbb R^+\longrightarrow \mathbb R$ be a continuous function. Then the application  $V:\mathcal S_0^n \longrightarrow \mathbb R$ given by 
$$V(K)=\int_{S^{n-1}} \theta(\rho_K(t)) dm(t)$$ is a radial continuous rotation invariant valuation. 
\end{teo}

We believe the result remains true if we remove the hypothesis that $V$ is positive and $V(\{0\})=0$, but at the moment we have not found a proof for this. Obviously, an analog of the result can be stated for negative valuations. 

We can define polynomial valuations as those induced by a polynomial on the star bodies (with respect to the radial sum), see Section \ref{polinomicas}. Note that this definition is {\em not} the exact analog of polynomial valuations for convex bodies defined in \cite{KhP}. 

Polynomial valuations on star bodies can be easily characterized using the results of \cite{JiVi}. Rotation invariant polynomial valuations turn out to be constant multiples of  dual quermassintegrals. With this characterization, and the  Stone-Weierstrass Theorem, one obtains immediately the following corollary to Theorem \ref{main}. It is formally related to \cite[Theorem A]{Alesker}, and it can be considered as a weak form of a  dual Hadwiger's Theorem. 

\begin{corolario}\label{maincorolario}
Every radial continuous rotation invariant valuation $V:\mathcal S_0^n\longrightarrow \mathbb R^+$ with $V(\{0\})=0$ can be approximated uniformly on bounded subsets of $\mathcal S_0^n$ by dual quermassintegrals. 
\end{corolario}

Most of the paper is devoted to the proof of Theorem \ref{main}. This proof is somehow long and technical. We have not found a way to simplify it significantly. On the other hand, the statement of the result is very much related to close results in the area and it probably seemed a ``reasonable conjecture'' for  a long time now, at least since the publication of \cite{Klain96, Klain97}. 

In the following paragraphs we sketch a description of the proof. The reader can find all the proper definitions in Section \ref{sectionnotation}.

The proof will be done with functional analysis and measure theory techniques. The  link is provided by  the natural bijection between rotationally invariant radial continuous valuations  $V:\mathcal S_0^n\longrightarrow \mathbb R^+$ and applications $\tilde{V}:C(S^{n-1})^+\longrightarrow \mathbb R^+$  which are rotationally invariant, continuous,  and verify $\tilde{V}(f)+ \tilde{V}(g)=\tilde{V}(f\vee g)+ \tilde{V}(f\wedge g)$. Our goal will be to obtain an integral representation for these applications $\tilde{V}$. 

The approach is similar to the  proof of the Riesz Representation Theorem for the dual of a $C(K)$ space: we want to define a measure on the Borel sets of $S^{n-1}$ based on $\tilde{V}$. The difference, and the difficulties, arise from the fact that now $\tilde{V}$ is in general not linear. 

For a better understanding of the relation of our result and techniques with the results and techniques in  \cite{Klain96, Klain97}, note that in those papers the valuations are supposed to be defined on $L^n$-stars. As mentioned before,  an $L^n$-star is a  star set whose radial function belongs to $L^n(S^{n-1})$. In particular, for every Borel set $A\subset S^{n-1}$, the characteristic function $\chi_A$ defines an $L^n$-star. Therefore,  one can consider the set function defined on the Borel sets of $S^{n-1}$ which maps a set  $A$ to the number obtained by applying $V$ to  the star set whose radial function is $\chi_A$. It is easy to see that this set function is a measure.  In the case we study, since $\chi_A$ is continuous only in trivial cases,  this star set is not a star body, and we can not apply $V$ to it.

To define the measure in our case we must proceed in several steps.  First, for every $\lambda>0$ we can consider the restriction of $V$ to the radial bodies contained in $\lambda$ times the unit ball of $\mathbb R^n$. We construct an outer measure, and an associated measure, based on this restriction of $V$. The rotational invariance of $V$ translates into the rotational invariance of this measure, and therefore it will be a constant multiple of $m$, the Lebesgue measure on $S^{n-1}$. This construction is done in Section \ref{primeramedida}. This measure will {\em not} be the  one we are looking for. But it will allow us to guarantee that $V$ is continuous with respect to $m$ in the natural sense. 

Once we know that $V$ is continuous with respect to  $m$, for very $\lambda>0$ we can define a {\em content} based on $V$. This content will allow us to define a second measure associated to it. This second measure will also be rotationally invariant and, therefore, again a constant multiple of the Lebesgue measure. This construction is done in Section \ref{segundamedida}.

In Section \ref{demodemain} we prove that this second measure allows us to obtain the integral representation of Theorem \ref{main}. 

Finally, in Section \ref{polinomicas} we characterize polynomial valuations on star-bodies. We show that if they are rotationally invariant then they are constant multiples of the dual quermassintegrals and we prove Corollary \ref{maincorolario}.

\section{Notation and previous results}\label{sectionnotation}


 

A set $L\subset \mathbb R^n$ is {\em star shaped at $0$}, or, more simply, a {\em star set} if it contains the origin and every line through $0$ that meets $L$ does so in a (possibly degenerate) line segment. We denote by $\mathcal S^n$ the set of the star sets of $\mathbb R^n$. 

Given a star set $L$, we define its {\em radial function} $\rho_L$ by


$$\rho_L(x)=  \sup \{c\geq 0 \, : \, cx\in L\}.$$

Clearly, radial functions are totally characterized by their restriction to $S^{n-1}$, the euclidean unit sphere in $\mathbb R^n$, so from now on we consider them defined on $S^{n-1}$.

Conversely, given a positive function $f:S^{n-1}\longrightarrow \mathbb R^+=[0,\infty)$  there exists a star set $L_f$ such that $f$ is the radial function of $L_f$.

A star set $L$ is called a {\em star body} if and only if $\rho_L$ is continuous. We denote by $\mathcal S_0^n$ the set of star bodies. 

Given two sets $L,M\in \mathcal S^n$, we define their {\em radial sum} as the star set $L\tilde{+}M$ whose radial function is $\rho_L+\rho_M$. Note that the radial sum of two star bodies is again a star body.

In the space of convex bodies, the natural topology is given by the Hausdorff metric. Its analog for star sets and bodies is the {\em radial topology}, induced by the {\em radial metric}. The radial metric  is  defined by $$\delta(K,L)=\inf\{\lambda\geq 0 \mbox{ such that } K\subset L\tilde{+} \lambda B_n, L\subset K\tilde{+} \lambda B_n\}.$$

It is very easy to see that the radial metric can equivalently be defined by  $$\delta(K,L)=\|\rho_K-\rho_L\|_\infty.$$

\smallskip 

We say that an application $V:\mathcal S_0^n\longrightarrow \mathbb R$ is a {\em valuation} if, for every pair of star bodies $K,L$, $$V(K\cup L)+V(K\cap L)=V(K)+ V(L).$$

\smallskip 

Given two functions $f_1, f_2:S^{n-1}\longrightarrow \mathbb R$ we denote their supremum and infimum by $$(f_1\vee f_2)(t)=\sup \{f_1(t), f_2(t)\}$$
$$(f_1\wedge f_2)(t)=\inf \{f_1(t), f_2(t)\}.$$ 

Given two star bodies $K,L$, both $K\cup L$ and $K\cap L$ are star bodies, and it is easy to see that $\rho_{K\cup L}=\rho_K\vee \rho_L$ and $\rho_{K\cap L}=\rho_K\wedge \rho_L$.

Given a topological space $X$ and a set $A\subset X$, we denote the closure of $A$ by $\overline{A}$. 
Given  a function $f:X\longrightarrow \mathbb R$, we define the support of $f$ by $supp(f)=\overline{\{x\in X \mbox{ such that } f(x)\not = 0\}}$. Given  a function $f:X\longrightarrow \mathbb [0,1]$, an open set $G\subset X$, and a compact set $K\subset X$, we say that $f\prec G$ if $\supp(f)\subset G$ and we say that $K\prec f$ if $f(t)=1$ for every $t\in K$.

 $\uno:S^{n-1}\longrightarrow \mathbb R$ is the function constantly equal to 1. We denote the euclidean unit ball of $\mathbb R^n$ by $B_{\mathbb R^n}$.

We denote by $\Sigma_n$ the  Borel $\sigma$-algebra of $S^{n-1}$. That is, the smallest $\sigma$-algebra that contains the open sets of $S^{n-1}$. 

$S(\Sigma_n)$ denotes  the normed space of the simple functions over $\Sigma_n$, endowed with the supremum norm. $B(\Sigma_n)$ denotes its completion. $C(S^{n-1})$ is the space of continuous (real valued) functions defined on $S^{n-1}$. $C(S^{n-1})^+$ denotes the positive functions of $C(S^{n-1})$. It is well known that $C(S^{n-1})$ is naturally isometrically contained in $B(\Sigma_n)$. 
We will use $C(S^{n-1})^*, B(\Sigma_n)^*$ to denote the topological duals of $C(S^{n-1})$ and $B(\Sigma_n)$ respectively.

We say that the set function $\mu:\Sigma_n \longrightarrow \mathbb R$ is a {\em signed measure} if it is countably additive over disjoint sets.

If $\mu$ is positive, we will call it simply a measure. 





\section{Construction of an outer measure}\label{primeramedida}

Let $V$ be a valuation as in the hypothesis of Theorem \ref{main}. The first step towards our proof is the construction of an outer measure associated to the valuation $V$.

We can define an application  $\tilde{V}:C(S^{n-1})^+\longrightarrow \mathbb R^+$ associated to $V$ in the natural way: for every $f\in C(S^{n-1})^+$, we define $\tilde{V}(f)=V(L_f)$, where $L_f$ is the radial body associated to $f$. 

For every $\lambda>0$ we will construct an 
outer measure $\mu_{V,\lambda}^*$ associated to $V$. For simplicity in the notation we define the outer measure for the case $\lambda=1$, and we denote it just $\mu_V^*$. The extension to a general $\lambda>0$ is simple, and it is explicitly described at the end of this section. 

We use $\mathcal P(S^{n-1})$ for the set of the subsets of $S^{n-1}$. We recall that a set function  $\mu^*:\mathcal P(S^{n-1})\longrightarrow [0, +\infty]$ is an {\em outer measure} if  
\begin{itemize}
\item[(i)] \noindent $\mu^*(\emptyset)=0$
\item[(ii)] \noindent $\mu^*$ is monotone. That is, for every $A\subset B\subset S^{n-1}$, $\mu^*(A)\leq \mu^*(B)$.
\item[(iii)] \noindent $\mu^*$ is countably subadditive. That is, for every sequence  $(A_i)_{i\in \mathbb N}$ of sets in $\mathcal P(S^{n-1})$, $$\mu^*\left(\cup_{i\in \mathbb N} A_i\right)\leq \sum_{i\in \mathbb N} \mu^*(A_i).$$
\end{itemize}

\smallskip

We start defining our outer measure for open sets: 
For every open set $G\subset S^{n-1}$ we define $$\mu_1^*(G)=\sup \{\tilde{V}(f): \, f\prec G\}.$$
Now, for every $A\subset S^{n-1}$, we define \begin{equation}\label{defout}\mu_V^*(A)=\inf\{ \mu_1^*(G):\,  A\subset G, \, G \mbox{ an open set }\}.\end{equation}

It is very easy to see that, for every open set $G\subset S^{n-1}$, $\mu_1^*(G)=\mu^*_V(G)$:  It is clear that $\mu_V^*(G)\leq \mu^*_1(G)$ and the reverse inequality follows immediately after  noting that $\mu_1^*$ is  monotone on open sets.  

Therefore we drop the notation $\mu_1^*$ and we denote both by  $\mu_V^*$.

We have to check that $\mu_V^*$ is indeed an 
outer measure. First we need some observations. 


If $f_1, f_2$ are both continuous and positive, so are $f_1\vee f_2$ and $f_1\wedge f_2$. In this case,  if  $K_1, K_2$ are the star bodies associated to $f_1, f_2$,  then   $f_1\vee f_2$, $f_1\wedge f_2$ are the radial functions of  $K_1\cup K_2$ and $K_1\cap K_2$ respectively. 

Therefore, it follows from the definition of valuation that, for every $f_1, f_2\in C(S^{n-1})^+$,  
\begin{equation}\label{aditividad}
\tilde{V}(f_1)+\tilde{V}(f_2)=\tilde{V}(f_1\vee f_2)+ \tilde{V}(f_1\wedge f_2).
\end{equation}

Now, it is easy to prove by induction the following result, similar to the inclusion-exclusion principle: 

\begin{lema}\label{uniones}
Let $N\in \mathbb N$ and let $f_1, \ldots, f_N \in C(S^{n-1})^+$. Then $$V\left(\bigvee_{i=1}^N f_i\right)= \sum_{1\leq i\leq N} \tilde{V}(f_i) - \sum_{1\leq i_1<i_2\leq N} \tilde{V}(f_{i_1}\wedge f_{i_2}) $$ $$ +\sum_{1\leq i_1 < i_2 <i_3\leq N} \tilde{V}(f_{i_1}\wedge f_{i_2}\wedge f_{i_3})-\cdots + (-1)^{N-1} \tilde{V}(f_1\wedge f_2\wedge\cdots \wedge f_{N}).$$
\end{lema}

The good behaviour of linear functionals with respect to the sum of functions is replaced now by the good behaviour described in Lemma \ref{uniones} of valuations with  respect to the supremum of functions. For this reason, we will need  ``partitions of the unity'' through suprema, rather than sums. 

We say that a collection of subsets $\mathcal G$ of a topological space $X$ is {\em locally finite} if for any $x\in X$ there exists a neighborhood $U_x$ of $x$ such that $U_x$ intersects only finitely many subsets that belong to $\mathcal G$. 

Similarly, we say that a family $\{\varphi_i: i\in I\}$ of continuous functions $\varphi_i\in C(X)$ is locally finite if the family $\{\supp(\varphi_i): i\in I\}$ is locally finite.  

The following lemma is well known. We state it for completeness.

\begin{lema}\label{shrinking}
Let $X$ be a paracompact Hausdorff space (in particular $X$ can be a subset of $S^{n-1}$). Let $\{G_i: i\in I\}$ be an open cover of $X$. Then there exists a locally finite  open cover $\{V_i: i\in I\}$ such that $\overline{V}_i\subset G_i$ for every $i\in I$, where $\overline{V}$ denotes the closure of $V$. 
\end{lema}

We can now proceed similarly as in the case of the usual partitions of  unity and we can prove the next lemma. It is probably well known, but we have not found a reference for it. We state it in more generality than we actually need, since we will apply it in the case of finite families of open sets. 

\begin{lema}\label{partition}
Let $\{G_i: i\in I\}$ be a family of open subsets of $S^{n-1}$. Let $X=\cup_{i\in I} G_i$. Then, for every $i\in I$ there exists a function $\varphi_i: X \longrightarrow [0,1]$ continuous in $X$ verifying  $\varphi_i\prec G_i$ and such that $\bigvee_{i\in I} \varphi_i=\uno$ in $X$. 
\end{lema}
\begin{proof}
We apply Lemma \ref{shrinking} twice to the open cover $\{G_i: i\in I\}$ of the space $X$. Then we obtain two locally finite open covers of $X$, $\{V_i: i\in I\}$,  $\{W_i: i\in I\}$ verifying $$W_i\subset \overline{W}_i\subset V_i\subset \overline{V}_i\subset G_i,$$
for every $i\in I$. 

We apply now Urysohn's lemma and we obtain functions $\varphi_i:X\longrightarrow [0,1]$ continuous in $X$ and such that, for every $i\in I$, $\varphi_i=1$ in $\overline{W}_i$ and $\varphi_i=0$ in $V_i^c$. This completes the proof. 
\end{proof}



We need one more auxiliar result before we can prove that  $\mu_V^*$ is an outer measure. 

\begin{lema}\label{split}
Let $\{G_i: i\in I\}$ be a collection of open subsets of $S^{n-1}$. Let $f\in C(S^{n-1})^+$ verify $f\prec \bigcup_{i\in I}G_i$. Then, for every $i\in I$ there exists $f_i\in C(S^{n-1})^+$, with $f_i\prec G_i$, such that $\bigvee_{i\in I} f_i=f$. 
\end{lema}

\begin{proof}
Given $\{G_i: i\in I\}$ we construct $\{\varphi_i: i\in I\}$ as in Lemma \ref{partition}.  Now we define 

\begin{eqnarray*}
f_i(t) = \left \{ \begin{array}{cl}
      \displaystyle{ f(t)\varphi_i(t) 
      \quad}
        & \displaystyle{\quad \mbox{if } t\in \bigcup_{i\in I}G_i}\\
        \noalign{\bigskip}
 0 \quad & \quad \mbox{if } t\not \in \bigcup_{i\in I}G_i
 \end{array} \right.
\end{eqnarray*}

Clearly $f_i\prec G_i$ for every $i\in I$. For every $i\in I$, $f_i$ is continuous in $S^{n-1}$. To see this, note that $f_i$ is clearly continuous at $t$ if $t \in \bigcup_{i\in I}G_i$ or if $t\in \overline{\left(\bigcup_{i\in I}G_i\right)}^c$. Therefore, we only have to check continuity at the points $t$ in the boundary of $ \bigcup_{i\in I}G_i$. We fix one such $t$ and we consider a sequence $(t_k)_{k\in \mathbb N}\subset S^{n-1}$. We can divide this sequence into three subsequences: 
One in $\bigcup_{i\in I}G_i$, another one in $\overline{\left(\bigcup_{i\in I}G_i\right)}^c$ and the third in the boundary of $\bigcup_{i\in I}G_i$. It is clear that $f_i(t_j)$ converges to $f_i(t)=0$ along each of these subsequences. 

Finally, let  $t\in S^{n-1}$. If $t\in \bigcup_{i\in I}G_i$ then $\bigvee_{i\in I} f_i(t) = \bigvee_{i\in I} f(t)\varphi_i(t) =f(t) \bigvee_{i\in I} \varphi_i(t)=f(t)$. If $t\not \in \bigcup_{i\in I}G_i$ then $\bigvee_{i\in I} f_i(t)=0=f(t)$. 
\end{proof}







Now we can prove that  $\mu_V^*$ is an outer measure. 

\begin{prop}\label{medidaexterior}
Let $V:\mathcal S_0^n\longrightarrow \mathbb R^+$ be a radial continuous valuation verifying that $V(\{0\})=0$. Then $\mu_V^*$ defined as in Equation (\ref{defout}) is an outer measure. 
\end{prop}
\begin{proof}
Note first that $\mu^*_V(\emptyset)=\tilde{V}(0)=V(\{0\})=0$. The monotonicity of $\mu^*_V$ is immediate.

We prove next the countable subadditivity. Let $(A_i)_{i\in \mathbb N}$ be a sequence of subsets of $S^{n-1}$. If $\sum_{i \in \mathbb N} \mu_V^*(A_i)=\infty$, then there is nothing to prove. So we may assume that $\mu_V^*(A_i)<\infty$ for every $i \in \mathbb N$. Let $\epsilon>0$. For every $i\in \mathbb N$, choose an open set $G_i$ such that $A_i\subset G_i$ and $\mu^*_V(A_i)>\mu_V^*(G_i)-\frac{\epsilon}{2^i}$. Choose now $f\in C(S^{n-1})^+$ such that $f\prec(\bigcup_{i\in \mathbb N} G_i)$ and $$\mu^*_V(\bigcup_{i\in \mathbb N} G_i)\leq \tilde{V}(f)+ \epsilon.$$ Since $\supp(f)$ is compact, there  exists $N\in \mathbb N$ such that $\supp(f)\subset \bigcup_{i=1}^N G_i$.

We apply Lemma \ref{split} to the collection $\{G_i:\, 1\leq i \leq N\}$ and we obtain functions $f_i$ ($1\leq i \leq N$) as in the lemma. 

It follows from Equation (\ref{aditividad}) and the positivity of $V$ that  $$\tilde{V}(f_1\vee f_2)=\tilde{V}(f_1)+\tilde{V}(f_2)-\tilde{V}(f_1\wedge f_2)\leq \tilde{V}(f_1)+\tilde{V}(f_2).$$

Reasoning by induction, we easily get that 
$$\tilde{V}(f)= \tilde{V}\left(\bigvee_{i=1}^N f_i \right) \leq \sum_{i=1}^N \tilde{V}(f_i).$$

Now we have $$\mu^*_V\left(\bigcup_{i\in \mathbb N} A_i\right )\leq \mu^*_V\left(\bigcup_{i\in \mathbb N} G_i\right )\leq \tilde{V}(f)+\epsilon \leq \sum_{i=1}^N \tilde{V}(f_i) +\epsilon \leq $$ $$\leq \sum_{i=1}^N \mu_V^*(G_i)+\epsilon \leq \sum_{i\in \mathbb N} \mu_V^*(G_i)+\epsilon\leq \sum_{i\in \mathbb N} \mu_V^*(A_i)+ 2\epsilon.$$

Since $\epsilon>0$ was arbitrary, this finishes the proof.
\end{proof}






Given an outer measure $\mu^*$, we say that a set $B\subset S^{n-1}$ is $\mu^*$-measurable if for every $A\subset S^{n-1}$, $$\mu^*(A)=\mu^*(A\cap B) + \mu^*(A\cap B^c).$$

It is well known (\cite[Theorem 1.3.4]{Cohn}) that the set of $\mu^*$ measurable sets is a $\sigma$-algebra. Moreover, $\mu^*$ restricted to that $\sigma$-algebra is a measure. 

\begin{prop}\label{medidaca}
The Borel $\sigma$-algebra of  $S^{n-1}$, $\Sigma_n$, is $\mu^*_V$ measurable. Therefore, if we define $\mu_V$ as the restriction of $\mu_V^*$ to $\Sigma_n$, then $\mu_V$ is a measure. 
\end{prop}
\begin{proof}
We just need to check that every open set $G\subset S^{n-1}$ is $\mu^*_V$ measurable. It follows from the subadditivity of $\mu_V^*$ that it suffices to check that, for every $A\subset S^{n-1}$, 

$$\mu^*_V(A)\geq \mu^*_V(A\cap G) + \mu^*_V(A\cap G^c).$$

If $\mu_V^*(A)=\infty$, then there is nothing to prove. So we may assume that $\mu_V^*(A)<\infty$.
We fix $A\subset S^{n-1}$ and $\epsilon >0$. There exists an open set $U$, with $A\subset U$, such that $\mu^*_V(U)\leq \mu_V^*(A)+\epsilon$. 

$U\cap G$ is an open set. We choose $f_1\prec (U\cap G)$ such that $\mu^*_V(U\cap G) \leq \tilde{V}(f_1)+ \epsilon$. We consider the compact set $K=\supp(f_1)\subset (U\cap G)$. Then $(U\cap G^c)\subset (U\cap K^c)$, and this last set is open. Choose now $f_2\prec (U\cap K^c)$ such that $\mu^*_V(U\cap K^c)\leq \tilde{V}(f_2)+\epsilon$.

Note that $f_1$ and $f_2$ have disjoint supports, both of them contained in $U$. Therefore $f_1\vee f_2 \prec U$,  $f_1\wedge f_2=0$, $\tilde{V}(f_1\vee f_2)=\tilde{V}(f_1)+ \tilde{V}(f_2)$ and we have

$$\mu^*_V(A)\geq \mu^*_V(U)-\epsilon \geq \tilde{V}(f_1\vee f_2)-\epsilon=  \tilde{V}(f_1)+ \tilde{V}(f_2)-\epsilon\geq $$ $$\geq \mu^*_V(U\cap G) + \mu^*_V(U\cap K^c)-3\epsilon \geq  \mu^*_V(U\cap G) + \mu^*_V(U\cap G^c)-3\epsilon \geq$$ $$\geq \mu^*_V(A\cap G) + \mu^*_V(A\cap G^c)-3\epsilon. $$
\end{proof}







So, we have seen that given a positive valuation $V$ on $\mathcal S_0^n$ we can associate to it in a natural way a measure $\mu_V: \Sigma_n\longrightarrow \mathbb [0,\infty]$.

It is immediate to see that if  $V$ is rotationally invariant,  so is $\mu_V$.

Let us see that $\mu_V$ is finite. Suppose that $\mu_V(S^{n-1})=\infty$. Let $G\subset S^{n-1}$ be any fixed nonempty open set, and choose $t_0\in G$. For every $t\in S^{n-1}$, let $\varphi_t$  be a rotation in $S^{n-1}$ such that $\varphi_t(t_0)=t$. Let $G_t$ be the open set $\varphi_t(G)$. Then $\cup_{t\in S^{n-1}}G_t$ is an open cover of the compact set $S^{n-1}$. Pick a finite subcover $G_{t_1},\ldots, G_{t_k}$. It follows from subadditivity together with rotational invariance that $$\infty=\mu_V(S^{n-1})\leq \sum_{i=1}^k \mu_V(G_{t_i})=k \mu_V(G),$$
and, therefore, for every nonempty open set $G\subset S^{n-1}$, $\mu_V(G)=\infty$. 

Let now $(H_i)_{i\in \mathbb N}$ be a sequence of nonempty disjoint open subsets of $S^{n-1}$. It follows from the previous paragraph that, for each $i\in \mathbb N$, we can consider $f_i\in C(S^{n-1})^+$ with  $f_i\prec H_i$ and  such that $\tilde{V}(f_i)\geq 1$. Note that, if $i\not = j$, $f_i\wedge f_j=0$. Let $f=\bigvee_{i\in \mathbb N} f_i$. It follows from the  valuation property that $V(f)$ is not in $\mathbb R^+=[0, +\infty)$, a contradiction.

Therefore $\mu_V$ is finite and rotational invariant and, hence, proportional to the Lebesgue measure. That is, there exists a positive $\vartheta$ such that $\mu_V=\vartheta m$, where $m$ is the Lebesgue measure in $S^{n-1}$. 

Similarly, given $\lambda>0$, we can repeat the procedure above and define an outer measure $\mu^*_{V,\lambda}$ on open sets by the formula  
$$\mu_{V,\lambda}^*=\sup \{\tilde{V}(f): \, \frac{f}{\lambda}\prec G\}.$$

Then, we can extend it to general sets $A$ and define a measure $\mu_{V,\lambda}$ as we did for $\mu_{V}$. As in the case of $\mu_V$, all of the $\mu_{V,\lambda}$ are rotationally invariant. Hence, for every $\lambda>0$ there exists $\vartheta_\lambda\geq 0$ such that 
\begin{equation}\label{vartheta}
\mu_{V,\lambda}=\vartheta_\lambda m.
\end{equation} 
With this notation,  $\mu_{V}=\mu_{V,1}$ and $\vartheta=\vartheta_1$.

If $V$ is increasing in the sense that for every continuous $f\in C(S^{n-1})^+$ and for every $\lambda \geq 1$ one has $\tilde{V}(f)\leq \tilde{V}(\lambda f)$, then 
$\mu_V$ can be used to obtain an integral representation of $V$. But $V$ need not be increasing, and we will require more involved reasonings. 

To make clear why $\mu_{V}$ does not properly capture the values of $V$, consider the function $\theta:\mathbb R^+ \longrightarrow \mathbb R^+$ defined by $\theta(\lambda)=\tilde{V}(\lambda \uno)$. Suppose, for instance, that $V$ is such that $$\max_{\lambda\in [0,1]}\{\theta(\lambda)\}=\theta\left(\frac{1}{2}\right)$$ and that $\theta$ is strictly decreasing in $(\frac{1}{2}, 1]$. We will see in the next sections (and it follows as a consequence of Theorem \ref{main}) that, in that case,  for an open set $G$, $\mu_V(G)$ can be arbitrarily well approximated by $\tilde{V}(f)$, where $f$ are functions with $\|f\|_\infty\leq \frac{1}{2}$. It will follow that $\mu_{V,\frac{1}{2}}=\mu_{V,1}$, but $\tilde{V}(\frac{1}{2}\uno)> \tilde{V}(\uno)$. That is, the measures $\mu_{V,\lambda}$ do not suffice to characterize $V$. 

But the construction of the measures $\mu_{V,\lambda}$ does yield the next observation, which will be used several times in the next section. 

We say that $V\ll m$ if,  for every $\lambda>0$ and  $\epsilon>0$, there exists  $\delta >0$ such that for every open set $G$, with $m(G)<\delta$, and for every $f\in C(S^{n-1})^+$, with $\|f\|_\infty\leq\lambda$ and $\supp(f)\subset G$, one has $\tilde{V}(f)<\epsilon$. 

\begin{obs}\label{Vpequenha} If $V$ is as in the hypothesis of Theorem \ref{main}, then $V\ll m$. With more detail, let $\lambda>0$, $\epsilon>0$,  and let  $G\subset S^{n-1}$ be an open set such that $m(G)\leq \epsilon$. Then, for every $f\in C(S^{n-1})^+$ such that $\supp(f)\subset G$ and $\|f\|_\infty\leq \lambda$, 
$$ \tilde{V}(f)\leq \vartheta_\lambda \epsilon,$$
where $\vartheta_\lambda$ is given by Equation (\ref{vartheta}).
\end{obs}

\section{Construction of the second measure}\label{segundamedida}

In this section we define the measures that will allow us to obtain the integral representation of Theorem \ref{main}. In the previous section we defined a measure ``from above'', starting with an outer measure. Now we will proceed ``from below'', starting with a {\em content}.

We recall that a {\em content} in $S^{n-1}$ is a non negative, finite, monotone set function defined on the class of the compact subsets of $S^{n-1}$ which is subadditive and additive on disjoint sets \cite[\S 53]{Halmos}. For every $\lambda>0$, we will define a {\em content} based on $V$. As we did in the previous section, for simplicity in the notation we make the construction first for $\lambda=1$. The generalization will again be obvious.   

Given a compact set $K\subset S^{n-1}$, we define 

\begin{equation}\label{defcontent}
\zeta(K)=\inf\{\tilde{V}(f):  K\prec f\}
\end{equation} 

We want to see that $\zeta$ is a content. First we need a lemma. 

\begin{lema}\label{lambdaG}
Let $K\subset S^{n-1}$ be a compact set and let $G\supset K$ be an open set. Then $\zeta(K)=\inf\{\tilde{V}(f): K\prec f \prec G\}:=\zeta_G(K)$
\end{lema}
\begin{proof}
One of the inequalities is trivial. We only need to check that $\zeta(K)\geq \zeta_G(K)$. To see this, we choose $\epsilon>0$. We pick now $f\in C(S^{n-1})^+$ with $K\prec f$ such that $\zeta(K)\geq \tilde{V}(f)-\epsilon$. The set  $C=\supp(f)\setminus G$ is closed (it could be empty, in that case the next reasonings are trivial). Therefore $C$ is compact, and $K\cap C=\emptyset$. Using the regularity of the Lebesgue measure, we pick an open set $H\supset C$ such that $m(H\setminus C)\leq \frac{\epsilon}{\vartheta}$. Therefore, $m(G\cap H)\leq m(H\setminus C)\leq \frac{\epsilon}{\vartheta}$. We apply now Lemma \ref{partition} to the open sets $G,H$ and we obtain the functions $\varphi_G, \varphi_H$. We define $f_G=f\varphi_G$ and $f_H=f\varphi_H$. We have that $f=f_G\vee f_H$ and  $\supp(f_G\wedge f_H)\subset G\cap H$. Therefore, Observation \ref{Vpequenha} tells us that  $\tilde{V}(f_G\wedge f_H)- \epsilon \leq 0$. So, we have $$\zeta(K)\geq \tilde{V}(f)-\epsilon =\tilde{V}(f_G\vee f_H)-\epsilon\geq \tilde{V}(f_G\vee f_H) -\epsilon+ \tilde{V}(f_G\wedge f_H)-\epsilon=$$ $$=\tilde{V}(f_G)+\tilde{V}(f_H) -2\epsilon \geq \tilde{V}(f_G)-2\epsilon \geq \zeta_G(K)-2\epsilon,$$
and our result follows. 
\end{proof}

\begin{lema}
$\zeta$ is a content.
\end{lema}
\begin{proof}

$\zeta$ is clearly non-negative and monotone. To see that $\zeta$ is finite, note first   $V(B_{\mathbb R^n})=\tilde{V}( \uno)<\infty$. Therefore, for every closed set $C\subset S^{n-1}$, $\zeta(C)\leq \tilde{V}(\uno)<\infty$. 

Let us see that it is subadditive. Let $K_1, K_2$ be compact. For $i=1,2$ let $f_i\in C(S^{n-1})^+$ be such that $K_i\prec f_i$. Then $(K_1\cup K_2)\prec (f_1\vee f_2) $ and $$\zeta(K_1\cup K_2)\leq \tilde{V}(f_1\vee f_2) \leq \tilde{V}(f_1\vee f_2)+ \tilde{V}(f_1\wedge f_2)=\tilde{V}(f_1)+ \tilde{V}(f_2). $$

It follows that $$\zeta(K_1\cup K_2)\leq \inf_{K_1\prec f_1} \tilde{V}(f_1)+ \inf_{K_2\prec f_2}\tilde{V}(f_2)= \zeta(K_1)+  \zeta(K_2).$$

We have to see now that if $K_1, K_2$ are disjoint compact sets, then $$\zeta (K_1\cup K_2)\geq \zeta(K_1) + \zeta(K_2).$$   

To see this, we first fix $\epsilon >0$. We choose two disjoint open sets $G_1,G_2$ containing $K_1, K_2$ respectively. $K_1\cup K_2$ is contained in the open set $G_1\cup G_2$. Therefore we can apply Lemma \ref{lambdaG} and we obtain $f\in C(S^{n-1})^+$ such that $ K_1\cup K_2\prec f\prec G_1\cup G_2$ and $\zeta(K_1\cup K_2)\geq \tilde{V}(f)-\epsilon$. We define $f_1, f_2$ as the restrictions of $f$ to $G_1, G_2$ respectively. Clearly $f_1$ and $f_2$ are continuous, $K_1\prec f_1$,  $K_2\prec f_2$ and $f_1\wedge f_2=0$. Therefore
$$\zeta(K_1\cup K_2)\geq \tilde{V}(f)-\epsilon=\tilde{V}(f_1)+\tilde{V}(f_2)-\epsilon \geq \zeta(K_1)+ \zeta(K_2)-\epsilon.$$



\end{proof}

Once we have a content, we can construct a regular measure associated to it in a standard manner (see \cite[\S 53]{Halmos}): We define first an inner content on open sets $G$ by $\mu_*(G)=\sup \{\zeta(K): K\subset G\}$. Next we define an outer measure on all the subsets of $S^{n-1}$ $$\nu^*(A)=\inf\{\mu_*(G): G\supset A\}$$ and finally we consider the measure $\nu$ defined as the restriction of $\nu^*$ to the $\nu^*$-measurable sets. $\nu$ is a regular  measure on $\Sigma_n$, the Borel sets of $S^{n-1}$.

In general $\nu$ is not an extension of the content. But if $\zeta$ is {\em regular} then we can guarantee that $\nu$ is an extension of $\zeta$, that is $\nu(K)=\zeta(K)$ for every compact set $K\subset S^{n-1}$ (\cite[\S 54]{Halmos}).

We recall that a content $\zeta$ is regular if, for every compact $K$,

\begin{equation}\label{defregular}
\zeta(K)=\inf\{ \zeta(D): K\subset D^{\circ}; D \mbox{ compact }\},\end{equation}
where $A^\circ$ denotes the interior of a set $A$.

We see next that the content $\zeta$ defined in Equation (\ref{defcontent}) is regular. 

\begin{prop}\label{regularidad}
$\zeta$ is a regular content. 
\end{prop}

\begin{proof}
We have to show that Equation (\ref{defregular}) holds. One of the inequalities follows immediately from the monotonicity of $\zeta$. 

For the other inequality, fix a compact set $K$ and $\epsilon>0$. Choose  $f\in C(S^{n-1})^+$ such that $K\prec f$ and $\tilde{V}(f)\leq \zeta(K)+\epsilon$. 

Using the fact that $\tilde{V}$ is continuous at $f$, we get the existence of $0<\delta<1$ such that $|\tilde{V}(f)-\tilde{V}(g)|\leq \epsilon$ whenever $\|f-g\|_{\infty}\leq \delta$. We consider the function $g= \uno \wedge 
 (1+\delta) f(t).$ Then $g$ is continuous, $\|g-f\|_\infty\leq \delta$,  and $E\prec g$, where $E=f^{-1}([1-\frac{\delta}{2}, 1])$. Note that $E$ is compact and that 
$$K \subset f^{-1}\left(\{1\}\right) \subset f^{-1}\left( (1-\frac{\delta}{4}, 1]\right) \subset \left(f^{-1}\left([1-\frac{\delta}{2}, 1]\right)\right)^\circ=E^\circ.$$

Therefore $$\inf\{ \zeta(D): K\subset D^{\circ}; D \mbox{ compact }\}\leq  \zeta(E) \leq \tilde{V}(g) \leq \tilde{V}(f)+\epsilon \leq \zeta(K)+ 2\epsilon,$$
and our result follows. 
\end{proof}

It follows now from \cite[\S 54 Theorem A]{Halmos} that $\nu(K)=\zeta(K)$ for every compact set $K$.

Again, it follows from the fact that $V$ is rotationally invariant that $\zeta$, and therefore also $\nu$, are rotationally invariant. Therefore, we know that $\nu$ is proportional to $m$, the Lebesgue measure on $S^{n-1}$.  

In general, for every strictly positive  real number $\lambda>0$ we can define a content $\zeta_\lambda$ by \begin{equation}\label{contenidos}\zeta_\lambda(K)=\inf\{\tilde{V}(f): \mbox{ where } f\in C(S^{n-1})^+, \, \, K\prec \frac{f}{\lambda} \}. \end{equation}

So, our previous $\zeta$ becomes $\zeta_1$. For every $\lambda>0$, $\zeta_\lambda$ is a regular content with associated measure $\nu_\lambda$.

\section{Proof of the main result}\label{demodemain}

In this section we use the previous constructions to prove Theorem \ref{main}. 

For clarity in the exposition, we isolate in the next lemma a technical aspect of the proof.

\begin{lema}\label{pequenho}
Let $K\subset S^{n-1}$ be a compact set, let $\lambda>0$, let $\epsilon>0$ and let $\vartheta_\lambda$ be as in Observation \ref{Vpequenha}. Then, for every open set $G\supset K$  such that $m(G\setminus K)\leq \frac{\epsilon}{4\vartheta_\lambda}$ and for every  $f\in C(S^{n-1})^+$ such that $K\prec \frac{f}{\lambda}\prec G$,  $$\tilde{V}(f)\leq \nu_\lambda(K)+\epsilon.$$ 
\end{lema}
\begin{proof}For simplicity in the notation we write the proof for the case $\lambda=1$, and we just write $\vartheta, \nu$ for $\vartheta_1, \nu_1$. The general case is totally analogous. Let $K, \epsilon, G, f$ be as in the statement. $\tilde{V}$ is continuous at $f$. Therefore there exists $\delta_1$ such that $\tilde{V}(f)\leq \tilde{V}(h)+\frac{\epsilon}{4}$ whenever $\|f-h\|_\infty \leq \delta_1$. 

Let us consider the continuous function $\tilde{f}=\uno \wedge (1+\delta_1) f$ and the open set $G_1=f^{-1}\left((\frac{1}{1+\delta_1}, 1]\right)$. We apply Lemma \ref{lambdaG} to obtain a function $g$ such that $\tilde{V}(g)\leq \nu(K)+\frac{\epsilon}{4}$,  and  $ K\prec g\prec G_1$. This last fact implies that $g\leq \tilde{f}$. We use now the fact that $\tilde{V}$ is continuous at $\tilde{f}$ to obtain a $\delta_2$ such that $\tilde{V}(\tilde{f})\leq \tilde{V}(h)+\frac{\epsilon}{4}$ whenever $\|\tilde{f}-h\|_\infty \leq \delta_2$.

We define now $$C=\supp(f)\setminus g^{-1}\left(\left(\frac{1}{1+\delta_2}, 1\right]\right)$$ and $$H=G\setminus g^{-1}\left(\left[\frac{1}{1+\frac{\delta_2}{2}}, 1\right]\right).$$

$C$ is a closed set contained in the open set $H$. $C$ and $H^c$ are disjoint compact sets. So, we can choose disjoint open sets $U\supset C$, $W\supset H^c$ and we have that $\overline{U}\subset H$. We apply Urysohn's Lemma to the disjoint closed sets $C$ and $U^c$ and we obtain a  function $\varphi\in C(S^{n-1})^+$  verifying $C\prec \varphi \prec H$. 

We have  that $$\|\left(g\vee \varphi \tilde{f}\right)-\tilde{f}\|_\infty\leq 1-\frac{1}{1+\delta_2}< \delta_2.$$

To see this, note that if $t\not \in \supp(f)$ then $\tilde{f}(t)=g(t)=0$. If $t\in C$, then $\varphi(t)=1$ and, hence, $$(\varphi(t) \tilde{f}(t))\vee g(t)= \tilde{f}(t)\vee g(t)=\tilde{f}(t),$$
where the last equality follows from the fact that $g\leq \tilde{f}$.

Finally, if $t\in \supp(f)\setminus C$ then $\tilde{f}(t)=1$ and $\frac{1}{1+\delta_2}<g(t)\leq 1$. 

\smallskip

Note that $H\subset G\setminus K$ and therefore $m(H)\leq m(G\setminus K)\leq \frac{\epsilon}{4\vartheta}$. Hence, since $\varphi\tilde{f}\prec H$, Observation \ref{Vpequenha}  implies that $\tilde{V}(\varphi\tilde{f})\leq \frac{\epsilon}{4}$.

Finally, using the fact that $\|f-\tilde{f}\|_\infty < \delta_1$, we get $$\tilde{V}(f)\leq \tilde{V}(\tilde{f}) + \frac{\epsilon}{4} \leq \tilde{V}\left(g\vee \varphi \tilde{f}\right) + \frac{\epsilon}{2}=$$ $$=\tilde{V}(g)+ \tilde{V}(\varphi \tilde{f})- \tilde{V}\left(g\wedge \varphi \tilde{f}\right)+ \frac{\epsilon}{2}\leq  \tilde{V}(g)+ \tilde{V}(\varphi \tilde{f}) + \frac{\epsilon}{2}\leq \nu(K)+ \epsilon.$$

\end{proof}






Now we can prove our main result.

\begin{proof}[Proof of  Theorem \ref{main}]
We prove first the first statement of the Theorem. Let  $V$ be as in the hypothesis.  
We consider the family of measures $\nu_\lambda$ defined in the previous section. 

For every $\lambda\geq 0$,  $\nu_\lambda$ is rotationally invariant and, hence, proportional to the Lebesgue measure. Let us call $\theta(\lambda)$ the positive number that verifies $\nu_\lambda=\theta(\lambda) m$, and we define $\theta(0)=0$.

Then, recalling that $V(\{0\})=0$, for every $\lambda\geq 0$ we have $$\theta(\lambda)m(S^{n-1})=\nu_\lambda(S^{n-1})=\tilde{V}(\lambda\uno)=V(\lambda B_{\mathbb R^n}).$$

Therefore, it follows from the continuity of $V$ that the function \begin{equation}\label{theta}\theta: [0, \infty) \longrightarrow [0, \infty)\end{equation}
defined by $$\lambda\mapsto \theta(\lambda)$$ 
is continuous.

We consider a function $f\in C(S^{n-1})^+$. We want to see that $$\tilde{V}(f)=\int_{S^{n-1}} \theta(f(t)) dm(t).$$

For a given   $\delta>0$,  let $N=\left[\frac{\|f\|_\infty}{\delta}\right]+1$, where $[a]$ is the integer part of $a$.

For every $l\in \mathbb N$ we define $D_l=\{\lambda: \, m\left(f^{-1}(\{\lambda\})\right)\geq \frac{1}{l}\}$. Since $m(S^{n-1})<\infty$, we have that $D_l$ is finite and therefore $D=\bigcup_{l\in \mathbb N} D_l$ is at most numerable. As a result, we get that for every $1\leq i \leq N$ there exists a $\delta_i\in \mathbb R$ such that $|\delta_i-i\delta|<\frac{\delta}{100}$ and $m\left(f^{-1}(\{\delta_i\})\right)=0$. We pick $\delta_N$ such that it addtionally verifies $\delta_N\geq\|f\|_\infty$.

We define $A_1=f^{-1}\left([0, \delta_1)\right)$ and for every $2\leq i \leq N$ we define  $$A_i=f^{-1}\left((\delta_{i-1}, \delta_i)\right).$$

For $1\leq i \leq N$ we define also the sets $C_i=f^{-1}(\{\delta_i\})$.

Next, we consider the  simple function $g_\delta:S^{n-1}\longrightarrow \mathbb R^+$ defined by $$g_\delta(t)=\sum_{i=1}^{N} \delta_i \chi_{A_i\cup C_i}(t).$$

For every $\delta>0$, $\|g_\delta - f\|_\infty\leq 2\delta$. Since $\theta$ is uniformly continuous in $[0,\|f\|_\infty+1]$, we get that $\|\theta(g_\delta)-\theta(f)\|_\infty$ converges to $0$ as $\delta$ converges to $0$. That is, $\theta(g_\delta)$ converges to $\theta(f)$ in the norm topology of $B(\Sigma_n)$, the bounded Borel functions on $S^{n-1}$. Therefore, since the application $g\mapsto \int_{S^{n-1}} gdm$ belongs to $B(\Sigma_n)^*$, we get that 
$$\lim_{\delta\rightarrow 0} \int_{S^{n-1}} \theta(g_\delta(t)) dm(t)=\int_{S^{n-1}} \theta(f(t)) dm(t).$$  

We note that $$\int_{S^{n-1}} \theta(g_\delta(t)) dm(t)=\sum_{i=1}^{N} \theta(\delta_i)m(A_i).$$

Therefore, to prove the first part of the Theorem is suffices to check that for every $\epsilon>0$ there exists $\Delta>0$ such that for every $\delta<\Delta$, $$\sum_{i=1}^{N} \theta(\delta_i)m(A_i)- 3\epsilon\leq \tilde{V}(f) \leq \sum_{i=1}^{N} \theta(\delta_i)m(A_i) + 3\epsilon.$$

Using the  definition of $\nu_\lambda$ and $\theta$, we can write the previous inequality as 
\begin{equation}\label{desigualdad}\sum_{i=1}^{N} \nu_{\delta_i}(A_i)- 3\epsilon\leq \tilde{V}(f) \leq \sum_{i=1}^{N} \nu_{\delta_i}(A_i) + 3\epsilon.
\end{equation}

First we check the first inequality. We fix $\epsilon>0$. 
Since $\tilde{V}$ is continuous at $f$, there exists $\Delta>0$ such that $|\tilde{V}(\tilde{f})-\tilde{V}(f)|\leq \epsilon$ whenever $\|\tilde{f}- f\|\leq 2\Delta$. We pick $\delta<\Delta$ and we define $N, \delta_i, g_\delta$ as above.

For every $1\leq i \leq N$, $A_i$ is open and  $\nu_{\delta_i}$, is regular. Therefore, we can  choose a compact set $K_i\subset A_i$ 
such that  $\nu_{\delta_i}(A_i)\leq \nu_{\delta_i}(K_i)+\frac{\epsilon}{N}$.  Note that $K_i\cap K_j=\emptyset$ if $i\not = j$. 

$K_1, \ldots, K_N$, $C_1, \ldots, C_N$ are disjoint compact sets. We consider pairwise disjoint open sets $V_1,\ldots, V_N$, $H_1, \ldots, H_N$ such that for every $1\leq i \leq N$, $K_i\subset V_i$ and $C_i\subset H_i$. 

Since $m(C_i)=0$, we may choose $H_i$ such that $m(H_i)\leq \frac{\epsilon}{N^2 \vartheta_{\delta_N}}$, where  $\vartheta_{\delta_N}$ is defined as in Observation \ref{Vpequenha}.

$A_i^c$ and $K_i$ are disjoint compact sets. Then we can take disjoint open sets $V'_i\supset K_i$, $W_i\supset A_i^c$. For every $1\leq i \leq N$ we define now $U_i=V_i\cap V'_i$ and we have $K_i\subset U_i\subset \overline{U}_i\subset A_i$. 

We now use Urysohn's Lemma again and, for every $1\leq i \leq N$, we can  consider a function $\psi_i\in C(S^{n-1})^+$ such that $\psi_i(t)=\delta_i$ for every $t\in K_i$,  $\psi_i(t) \leq \delta_i$ for every $t\in S^{n-1}$, and $\psi_i(t)=0$ for every $t\in U^c_i$.

We define now $$\tilde{f}=f\vee\left (\bigvee_{i=1}^N \psi_i\right).$$

Then $\tilde{f}(t)=\delta_i$ for every $t\in K_i$ and $\|\tilde{f}-f\|_\infty\leq 2\delta$.

We define $A_0=\cup_{i=1}^N H_i$. Then $m(A_0)\leq \frac{\epsilon}{N \vartheta_{\delta_N}}$ and the collection $\{A_i: 0\leq i \leq N\}$ is an open cover of $\supp (f)$ .



Now we apply Lemma \ref{partition} to the family $\{A_i: 0\leq i \leq N\}$ and we obtain the functions $\varphi_i$  ($0\leq i \leq N$) as in that lemma. We note that for $1\leq i \leq N$, $K_i\prec \varphi_i$.

We define the functions $\tilde{f}_i=\tilde{f}\varphi_i$, $0\leq i \leq N$.  As in the proof of Lemma \ref{split} we see  that they are all continuous and $\tilde{f}=\bigvee_{i=0}^N \tilde{f}_i$. 
They also verify $K_i\prec \frac{\tilde{f}_i}{\delta_i}$. Therefore, $\nu_{\delta_i}(K_i)\leq \tilde{V}(\tilde{f}_i)$ for $1\leq i \leq N$.

It follows from Lemma \ref{uniones} and the fact that $\tilde{f}_i\wedge \tilde{f}_j=0$ for every $1\leq i < j\leq N$ that  

$$\sum_{i=0}^N \tilde{V}(\tilde{f}_i)= \tilde{V}\left(\bigvee_{i=0}^N \tilde{f}_i\right)+ \sum_{i=1}^{N} \tilde{V}(\tilde{f}_{i}\wedge \tilde{f}_{0}).$$

Note that $\supp (\tilde{f}_i\wedge \tilde{f}_0)\subset A_i\cap A_0$ and $m(A_i\cap A_0)\leq m(A_0)\leq   \frac{\epsilon}{N \vartheta_{\delta_N}}$. Therefore, Observation \ref{Vpequenha} guarantees that $\sum_{i=1}^N \tilde{V}(\tilde{f}_i\wedge \tilde{f}_0)\leq \epsilon$. Hence

$$\sum_{i=0}^N \tilde{V}(\tilde{f}_i)  \leq \tilde{V}\left(\bigvee_{i=0}^N \tilde{f}_i\right)+ \epsilon = \tilde{V}(\tilde{f})+\epsilon \leq \tilde{V}(f)+2\epsilon. $$

Putting things together, we have 

$$\sum_{i=1}^{N} \theta(\delta_i)m(A_i)=\sum_{i=1}^{N} \nu_{\delta_i}(A_i)\leq \sum_{i=1}^{N} \nu_{i\delta}(K_i)+ \epsilon \leq \sum_{i=1}^{N} \tilde{V}(\tilde{f}_i)+ \epsilon\leq $$ $$\leq \sum_{i=0}^{N} \tilde{V}(\tilde{f}_i)+ \epsilon \leq \tilde{V}(f)+3\epsilon.$$

\bigskip

We prove now the second inequality in (\ref{desigualdad}). We let $\epsilon$, $\Delta$, $ \delta$, $ A_i$, $ K_i$, $ \varphi_i$, $ \tilde{f}$, $ \tilde{f}_i$ be as in the first part of the proof, with the following extra condition: For $1\leq i \leq N$, we require that   $$m(A_i\setminus K_i)\leq \frac{\epsilon}{8N \vartheta_{\delta_N}}.$$

Clearly the $K_i$'s can be chosen to meet this additional requirement. 
Note that now we have  $K_i \prec \frac{\tilde{f}_i}{\delta_i} \prec A_i$ and $m(A_i\setminus K_i)\leq \frac{\epsilon}{8N \vartheta_{\delta_N}}$. Therefore, Lemma \ref{pequenho} implies that $\tilde{V}(\tilde{f}_i)\leq \nu_{\delta_i}(K_i)+ \frac{\epsilon}{2N}.$

Since  $m(A_0)<\frac{\epsilon}{N\vartheta_{\delta_N}}$,
Observation \ref{Vpequenha} implies that $\tilde{V}(\tilde{f}_0)\leq \frac{\epsilon}{N}\leq \epsilon$.

So, we have  $$\tilde{V}(f)\leq \tilde{V}(\tilde{f}) + \epsilon=\tilde{V}\left(\bigvee_{i=0}^N \tilde{f}_i\right)+ \epsilon=\sum_{i=0}^N \tilde{V}(\tilde{f}_i) - \sum_{i=1}^{N} \tilde{V}(\tilde{f}_i\wedge \tilde{f}_0) + \epsilon\leq $$ $$\leq \sum_{i=1}^{N} \tilde{V}(\tilde{f}_i)+ \tilde{V}(\tilde{f}_0) + \epsilon  \leq \sum_{i=1}^{N} \nu_{\delta_i}(K_i) +  3\epsilon  \leq  \sum_{i=1}^{N}\theta(\delta_i) m(A_i) + 3\epsilon.$$


\

\bigskip

We prove now the second half of the statement. Let $\theta$ be as in the hypothesis, and, for every star body $K$, define $$V(K)=\int_{S^{n-1}} \theta(\rho_K(t)) dm(t).$$

It is immediate that $V$ is well defined. Let us see that it is a valuation. Let $K, L$ be star bodies. Let 

$$C_1=\{t\in S^{n-1}: \rho_K(t) \geq \rho_L(t)\}$$
$$C_2=S^{n-1}\setminus C_1.$$

Using again that  $\rho_{K\cup L}=\rho_K \vee \rho_L $ and $\rho_{K\cap L}= \rho_K \wedge \rho_L$, we have 
 $$V(K\cup L) + V(K\cap L) =\int_{S^{n-1}} \theta(\rho_{K\cup L}(t)) dm(t) +
\int_{S^{n-1}} \theta(\rho_{K\cap L}(t)) dm(t)=$$ $$=\int_{C_1} \theta(\rho_{K}(t)) dm(t) +\int_{C_2} \theta(\rho_{ L}(t)) dm(t)+ $$ $$+\int_{C_1} \theta(\rho_{L}(t)) dm(t)+ \int_{C_2} \theta(\rho_{K}(t)) dm(t)=V(K)+V(L).$$

\smallskip

Let us see that $V$ is continuous. Let $K$ be a radial body and let $(K_i)_{i\in \mathbb N}$ be a sequence of radial bodies converging to $K$ in the radial metric. As we mentioned before,   $K_i$ converges to $K$ in the radial metric if and only if $\rho_{K_i}$ converges to $\rho_K$ in the uniform norm. It follows from the compactness of $S^{n-1}$ that $\rho_K$ is bounded. So, there exists a closed bounded interval $I\subset \mathbb R$ such that $\rho_K(t), \rho_{K_i}(t) \in I$ for every $t\in S^{n-1}$, $i\in \mathbb N$. Now, $\theta$ is uniformly continuous in $I$, and it follows immediately that $\theta(\rho_{K_i})$ converges to $ \theta(\rho_{K})$ in the uniform norm. Using again that the linear mapping $f\mapsto \int_{S^{n-1}} f dm$ belongs to $C(S^{n-1})^*$, we have that 

$$V(K_i)=\int_{S^{n-1}}\theta(\rho_{K_i}(t)) dm(t)\rightarrow \int_{S^{n-1}} \theta(\rho_{K}(t)) dm(t) =V(K).$$

\smallskip

The fact that  $V$ is rotation invariant follows immediately from the rotational invariance of $m$:  Let $K\in \mathcal S^n$ and let $\varphi$ be a rotation in $S^{n-1}$. Note that $\rho_{\varphi(K)}(t)=\rho_{K}(\varphi^{-1}(t))$. We have 

$$V(\varphi(K))=\int_{S^{n-1}} \theta(\rho_{\varphi(K)}(t)) dm(t)=\int_{S^{n-1}} \theta(\rho_{K}(\varphi^{-1}(t))) dm (t) = $$ $$ =\int_{S^{n-1}} \theta(\rho_{K}(t)) dm(t)=V(K).$$

\end{proof}

\section{Polynomial valuations  and dual quermassintegrals}\label{polinomicas}

In this section we will define polynomial valuations on convex bodies and we characterize them using the results of \cite{JiVi}. We show their connection with the dual quermassintegrals and  we  prove Corollary \ref{maincorolario}. 

We say that an application $T: (\mathcal S^n_0)^m\longrightarrow \mathbb R $ is a $k$-linear application if $T$ is  separately additive and positively homogeneous. This means that for every $L, L', L_2, \ldots, L_k \in \mathcal S_0^n$ and for every $\alpha>0, \beta>0$, $$T(\alpha L \tilde{+} \beta L', L_2, \ldots, L_k)=\alpha T(L, L_2, \ldots, L_k) +\beta T(L', L_2, \ldots, L_k),$$
where the role played by the first variable could also be played by any of the other variables.

We say that an application $P:\mathcal S^n_0\longrightarrow \mathbb R$ is a {\em $k$-homogeneous  polynomial} if there exists a $k$-linear application $T:(\mathcal S^n_0)^k\longrightarrow \mathbb R $ such that for every $L\in \mathcal S^n_0$, $P(L)=T(L, \ldots, L)$. 

We say that a valuation $V:\mathcal S_0^n\longrightarrow \mathbb R$ is a $k$-homogeneous {\em polynomial valuation} if $V$ is a $k$-homogeneous polynomial.  


The following result follows immediately from \cite[Theorem 3.4]{JiVi}. 

\begin{prop}\label{caracpolinomicas}
Let $V:\mathcal S_0^{n-1}\longrightarrow \mathbb R$ be a radial continuous $k$-homogeneous polynomial valuation. Then there exists a regular signed measure $\mu: \Sigma_n\longrightarrow \mathbb R$ such that, for every $L\in \mathcal S_0^n$, $$V(L)=\int_{S^{n-1}} (\rho_{L}(t))^k d\mu(t). $$

Conversely, for every regular signed measure $\mu: \Sigma_n\longrightarrow \mathbb R$, the above integral expression defines a radial continuous $k$-homogeneous polynomial valuation. 

Moreover, $V$ is rotationally invariant if and only if there exists a constant $c\in \mathbb R$ such that $\mu=cm$, where $m$ is the Lebesgue measure in $S^{n-1}$. 
\end{prop}

\begin{remark}
Note that this says that the only radial continuous $k$-homogenous rotationally invariant polynomial valuations are the constant multiples of the corresponding {\em dual quermassintegral} $\tilde{W}_{n-k}$ (see \cite{Gabook} for the definition). 
\end{remark}

The proof of Corollary \ref{maincorolario} follows now  easily:

\begin{proof}[Proof of Corollary \ref{maincorolario}] Let $V$ be as in the hypothesis. Let $B\subset \mathcal S_0^n$ be a bounded set. That is, there exists $M>0$ such that for every star body $K\in B$, and for every $t\in K$, $\|t\|\leq M$. Equivalently, for every $K\in B$, $\|\rho_K\|_\infty\leq M$. 

Let $\theta:\mathbb R\longrightarrow \mathbb R^+$ be the function associated to $V$ by Theorem \ref{main}.  It follows from the  Stone-Weierstrass Theorem that for every $\epsilon>0$ there exists $l\in \mathbb N$ and real numbers $a_0, \ldots, a_l$ such that for every $\lambda\in [0,M]$, $$|\theta(\lambda)-\sum_{k=0}^l a_k \lambda^k|<\epsilon.$$

For every $0\leq k \leq l$ we define the polynomial valuation $P_k$ by $$P_k(K)=a_k\int_{S^{n-1}} \rho_K^k(t) dm(t)=a_k n \tilde{W}_{n-k}.$$ 

Then we have that, for every $K\in B$, $$\left|V(K)- \sum_{k=0}^l a_k n \tilde{W}_{n-k}(K)\right|=$$ $$=\left|\int_{S^{n-1}} \theta(\rho_K(t)) dm(t) - \int_{S^{n-1}}\sum_{k=0}^l a_k \rho_K^k(t) dm(t)\right| \leq $$ $$\leq \int_{S^{n-1}}\left| \theta(\rho_K(t)) - \sum_{k=0}^l a_k \rho_K^k(t) \right| dm(t) \leq \epsilon m(S^{n-1}).$$ 
\end{proof}

\begin{centerline} {\em Acknowledgments} \end{centerline}
\smallskip

We would like to  thank  Marco Castrill\'on, Antonio Su\'arez Granero and Pedro Tradacete for helpful discussions. We also would like to thank the anonymous referee for an extremely careful scrutiny of the first version of the manuscript and for his or her  suggestions.

\end{document}